\documentclass[11pt,a4paper,reqno]{amsart}
\usepackage{amsmath,amsthm,amsfonts,amssymb,bm,wasysym}
\usepackage{caption}
\usepackage{subcaption}
\usepackage{fullpage}
\usepackage{graphicx}
\usepackage{tikz,ifthen}
\usepackage{hyperref}
\usetikzlibrary{calc}
\usetikzlibrary{patterns}
\usetikzlibrary{arrows}
\usetikzlibrary{decorations.pathreplacing}
\theoremstyle{plain}
\newtheorem{theorem}{Theorem}
\newtheorem{proposition}[theorem]{Proposition}
\newtheorem{corollary}[theorem]{Corollary}
\newtheorem{lemma}[theorem]{Lemma}

\theoremstyle{definition}

\theoremstyle{remark}

\begin{document}

\title{Bounded-hop percolation}
\def\MLine#1{\par\hspace*{-\leftmargin}\parbox{\textwidth}{\[#1\]}}

\def\A{\mathbb{A}}
\def\Ab{\mathcal{A}b}
\def\absq{{a^{\prime}}^2+{b^\prime}^2}
\def\AP{\text{G}}
\def\app{{a^{\prime\prime}}^2+1}
\def\argmin{\text{argmin}}
\def\arb{arbitrary }
\def\ass{assumption}
\def\arrow{\rightarrow}
\def\BT{B^{\mathbb{T}_n}}
\def\codim{\text{codim}}
\def\const{c}
\def\CCG{\text{G}}
\def\colim{\text{colim}}
\def\cond{condition }
\def\C{\mbox{\bf C}}
\def\ct{\mathsf{ct}}
\def\d{{\rm d}}
\def\dell{\partial}
\def\diam{\text{diam}}
\def\E{\mathbb{E}}
\def\envi{\mathsf{env}}
\def\enviIn{\partial^{\mathsf{in}}}
\def\enviOut{\partial^{\mathsf{out}}}
\def\enviInn{\partial^{\mathsf{in},*}}
\def\enviStab{\mathsf{env}_{\mathsf{stab}}}
\def\Et{\text{Et}}
\def\es{\emptyset}
\def\exp{\text{exp}}
\def\fa{for all }
\def\Fk{\mathcal{F}_{k_0}}
\def\Fm{Furthermore}
\def\G{\mathbb{G}}
\def\gr{\text{gr}}
\def\hge{h_{\mathsf{g}}}
\def\hco{h_{\mathsf{c}}}
\def\H{\text{H}}
\def\Hom{\text{Hom}}
\def\Hs{\widetilde{X}_{H,0}}
\def\inj{\hookrightarrow}
\def\id{\text{id}}
\def\iiets{it is easy to see }
\def\iietc{it is easy to check }
\def\Iietc{It is easy to check }
\def\Iiets{It is easy to see }
\def\imp{\Rightarrow}
\def\({\big(}
\def\){\big)}
\def\lver{\big|}
\def\rver{\big|}
\def\lcu{\big\{}
\def\rcu{\big\}}
\def\im{\mbox{im}}
\def\inn{\mathsf{in}}
\def\Inv{\text{Inv}}
\def\Ind{\text{Ind}}
\def\Ip{In particular}
\def\ip{in particular }
\def\LB{\text{LB}}
\def\Lo{\mathcal{L}^o}
\def\mc{\mathcal}
\def\mb{\mathbb}
\def\mf{\mathbf}
\def\Hp{\wt{X}_{H,0}^{'}}
\def\M{\mathbb{M}}
\def\Mo{Moreover}
\def\G{\mathbb{G}}
\def\GR{\mathsf{GR}}
\def\N{\mathbb{N}}
\def\Npo{\mathbf{N}_{\mathcal{P}^o}}
\def\k{\overline{k}}
\def\K{\underline{K}}
\def\LE{\mathsf{LE}}
\def\ldot{.}
\def\lmid{\;\middle\vert\;}
\def\O{\mathcal{O}}
\def\Ob{Observe }
\def\ob{observe }
\def\out{\mathsf{out}}
\def\Otoh{On the other hand}
\def\opartial{\partial^{\text{out}}}
\def\ipartial{\partial^{\text{in}}}
\def\eopartial{\partial^{\text{out}}_{\text{ext}}}
\def\eipartial{\partial^{\text{in}}_{\text{ext}}}
\def\p{\prime}
\def\pred{\mathsf{pred}}
\def\Trace{\mathsf{Trace}}
\def\pp{{\prime\prime}}
\def\Po{\mathcal{P}^0}
\def\P{\mathbb{P}}
\def\Proj{\mbox{\bf P}}
\def\Q{\mathbb{Q}}
\def\QQ{\overline{\Q}}
\def\pr{\text{pr}}
\def\R{\mathbb{R}}
\def\rstab{R_{\mathsf{stab}}}
\def\Spec{\text{Spec}}
\def\st{such that }
\def\sl{sufficiently large }
\def\ss{sufficiently small }
\def\sot{so that }
\def\su{suppose }
\def\succ{\mathsf{succ}}
\def\Su{Suppose }
\def\suf{sufficiently }
\def\udot{\mathaccent\cdot\cup}
\def\Set{\mathcal{S}et}
\def\T{\mathbb{T}}
\def\Twh{Then we have }
\def\Tes{There exists }
\def\te{there exist }
\def\tes{there exists }
\def\tptc{this proves the claim}
\def\Map{\text{Map}}
\def\VLo{\mc{VL}^o}
\def\wt{\widetilde}
\def\Wcon{We conclude }
\def\wcon{we conclude }
\def\wc{we compute }
\def\Wc{We compute }
\def\wo{we obtain }
\def\wh{we have }
\def\Wh{We have }
\def\Z{\mathbb{Z}}
\def\ZSlab{\mathbb{Z}^2_L\times\{0\}^{d-2}}

\author{Christian Hirsch}
\thanks{Weierstrass Institute Berlin, Mohrenstr. 39, 10117 Berlin, Germany; E-mail: {\tt hirsch@wias-berlin.de}.}

\begin{abstract}
Motivated by an application in wireless telecommunication networks, we consider a two-type continuum-percolation problem involving a homogeneous Poisson point process of users and a stationary and ergodic point process of base stations. Starting from a randomly chosen point of the Poisson point process, we investigate distribution of the minimum number of hops that are needed to reach some point of the second point process.
In the supercritical regime of continuum percolation, we use the close relationship between Euclidean and chemical distance to identify the distributional limit of the rescaled minimum number of hops that are needed to connect a typical Poisson point to a point of the second point process as its intensity tends to infinity. In particular, we obtain an explicit expression for the asymptotic probability that a typical Poisson point connects to a point of the second point process in a given number of hops.
\end{abstract}

\maketitle

\section{Introduction and main results}
\label{defSec}
We consider a model for a wireless telecommunication network where users are scattered at random in the entire Euclidean plane. In order to meet the users' communication demands, the operator sustains a network of base stations. In classical cellular networks, the base stations subdivide the plane into serving zones and all users inside a serving zone communicate directly with the associated base station. Although such networks exhibit a simple hierarchical topology, installation and upkeep are costly. Indeed, to guarantee  good quality of service to all users, the operator either needs to install (and maintain) a relatively dense network of base stations, or the base stations' transmission powers must be sufficiently high so that also distant users can be served.

Since the advent of LTE technology, operators have the possibility to reduce the number of required base stations substantially by using relays. As of today, this means installing fixed relays at locations that have been chosen in advance. For future generation networks it is desirable to extend this concept through the intelligent use of \emph{ad hoc technology}. To be more precise, we assume that each user has a (comparatively small) transmission radius. A direct communication between users is possible if they are within each others communication radii. Additionally, by forwarding messages via chains of directly connected users, base stations can communicate with distant users, even if transmission radii are comparatively small. 

Despite these virtues, having users act as relays entails a major drawback when it comes to quality of service for delay-sensitive applications. Indeed, the forwarding of messages via several hops induces substantial delay in message transmission. Hence, in network planning, it is crucial to have detailed knowledge of distributional properties of the minimum number of hops to a base station. 

In the random-graphs community, the minimum number of hops that are needed to connect two vertices of a graph is known as \emph{chemical distance}. In supercritical Bernoulli percolation on the lattice, chemical distance has been investigated in~\cite{antalPhD,antPiszt}. Loosely speaking, for distant points in the infinite connected component, the chemical distance is approximately proportional to the Euclidean distance, where the proportionality factor is called \emph{time constant}. The extension of this result to the setting of continuum percolation~\cite{yao} will be the major tool for establishing the distributional limit of the rescaled minimum number of hops needed to connect a user to a base station.

Next, we provide a precise definition of the wireless spatial telecommunication network under consideration. It consists of two types of network components. The first component is formed by network users. They are modeled by a homogeneous Poisson point process $X$ in $\R^d$, $d\ge2$ with some intensity $\lambda\in(0,\infty)$. The base stations constitute the second component. We assume that they are of the form $Y=rY^{(1)}$, where $Y^{(1)}$ is assumed to be a stationary and ergodic point process that is independent of $X$ and has a finite and positive intensity $\lambda'$. Here, $r\ge0$ is some scaling parameter controlling the intensity of base stations. Since we only assume stationarity and ergodicity, our results are valid under quite weak conditions on the spatial distribution of base stations. For instance, they can be applied to homogeneous Poisson point processes as well as randomly shifted lattices. In other words, our results do not depend on the question whether the base stations are scattered at random in the Euclidean plane or are aligned according to a grid that is viewed from a random reference point.

The random network under consideration can be thought of as a model for a wireless telecommunication network, where users can connect to base stations indirectly via at most $k\ge1$ hops of Euclidean distance at most $1$ to other network users. To be more precise, we say that $x,y\in\R^d$ are \emph{$k$-connectable} if there exist (not necessarily distinct) $X_{i_1},X_{i_2},\ldots,X_{i_{k-1}}\in X$ such that $|X_{i_j}-X_{i_{j+1}}|\le 1$ for all $j\in\{0,\ldots,k-1\}$, where $X_{i_0}=x$ and $X_{i_k}=y$. Here, $|\cdot|$ denotes the standard Euclidean norm in $\R^d$. We say that $x,y$ are \emph{connectable} if they are $k$-connectable for \emph{some} $k\ge1$. Figure~\ref{toyFig} shows a realization of the network model, where the points of $X$ and $Y$ are represented by dots and squares, respectively.
Points of $X$ that are $1$-connectable to some point of $Y$ are shown in blue, while points of $X$ that are $2$-connectable but not $1$-connectable to some point in $Y$ appear in green.

\begin{figure}[!htpb]
 \centering 
\begin{tikzpicture}[scale=0.85]
\fill (-2.1,-2.1) rectangle (-1.9,-1.9);\fill (-2.1,-0.1) rectangle (-1.9,0.1);\fill (-2.1,1.9) rectangle (-1.9,2.1);\fill (-0.1,-2.1) rectangle (0.1,-1.9);\fill (-0.1,-0.1) rectangle (0.1,0.1);\fill (-0.1,1.9) rectangle (0.1,2.1);\fill (1.9,-2.1) rectangle (2.1,-1.9);\fill (1.9,-0.1) rectangle (2.1,0.1);\fill (1.9,1.9) rectangle (2.1,2.1);
\fill (1.01,1.72) circle (2pt);
\fill (-0.6,2.91) circle (2pt);
\fill (2.96,1.58) circle (2pt);
\fill[color=green] (-1.37,2.66) circle (2pt);
\draw[color=green] (-1.37,2.66) -- (-1.68,2.58);
\fill[color=green] (-1.2,2.56) circle (2pt);
\draw[color=green] (-1.2,2.56) -- (-1.68,2.58);
\fill (2.55,-2.99) circle (2pt);
\fill (-0.49,-1.25) circle (2pt);
\fill (0.58,-0.72) circle (2pt);
\fill (-2.86,1.54) circle (2pt);
\fill[color=green] (-1.23,2.31) circle (2pt);
\draw[color=green] (-1.23,2.31) -- (-1.68,2.58);
\fill[color=green] (-1.23,2.31) circle (2pt);
\draw[color=green] (-1.23,2.31) -- (-1.53,1.74);
\fill (2.79,-1.61) circle (2pt);
\fill[color=green] (0.92,-1.19) circle (2pt);
\draw[color=green] (0.92,-1.19) -- (1.33,-1.71);
\fill[color=green] (1.13,2.87) circle (2pt);
\draw[color=green] (1.13,2.87) -- (0.49,2.42);
\fill (2.59,2.9) circle (2pt);
\fill[color=green] (-2.41,2.93) circle (2pt);
\draw[color=green] (-2.41,2.93) -- (-2.18,2.66);
\fill[color=green] (-0.69,-0.48) circle (2pt);
\draw[color=green] (-0.69,-0.48) -- (-0.15,-0.27);
\fill[color=green] (-2.03,-2.98) circle (2pt);
\draw[color=green] (-2.03,-2.98) -- (-1.74,-2.68);
\fill[color=green] (-2.03,-2.98) circle (2pt);
\draw[color=green] (-2.03,-2.98) -- (-2.27,-2.61);
\fill (0.96,0.09) circle (2pt);
\fill (1.14,0.86) circle (2pt);
\fill[color=blue] (0.35,2.47) circle (2pt);
\draw[color=blue] (0.35,2.47) -- (0,2);
\fill[color=blue] (0.49,2.42) circle (2pt);
\draw[color=blue] (0.49,2.42) -- (0,2);
\fill[color=blue] (0.19,-1.7) circle (2pt);
\draw[color=blue] (0.19,-1.7) -- (0,-2);
\fill[color=blue] (-1.79,0.05) circle (2pt);
\draw[color=blue] (-1.79,0.05) -- (-2,0);
\fill[color=blue] (-1.68,2.58) circle (2pt);
\draw[color=blue] (-1.68,2.58) -- (-2,2);
\fill[color=blue] (-0.1,0.64) circle (2pt);
\draw[color=blue] (-0.1,0.64) -- (0,0);
\fill[color=blue] (-1.52,0.48) circle (2pt);
\draw[color=blue] (-1.52,0.48) -- (-2,0);
\fill[color=blue] (0.21,-2.75) circle (2pt);
\draw[color=blue] (0.21,-2.75) -- (0,-2);
\fill[color=blue] (-0.05,-2.79) circle (2pt);
\draw[color=blue] (-0.05,-2.79) -- (0,-2);
\fill[color=blue] (-1.53,1.74) circle (2pt);
\draw[color=blue] (-1.53,1.74) -- (-2,2);
\fill[color=blue] (-2.18,2.66) circle (2pt);
\draw[color=blue] (-2.18,2.66) -- (-2,2);
\fill[color=blue] (-1.44,-2) circle (2pt);
\draw[color=blue] (-1.44,-2) -- (-2,-2);
\fill[color=blue] (-2.04,-1.48) circle (2pt);
\draw[color=blue] (-2.04,-1.48) -- (-2,-2);
\fill[color=blue] (-1.74,-2.68) circle (2pt);
\draw[color=blue] (-1.74,-2.68) -- (-2,-2);
\fill[color=blue] (-0.15,-0.27) circle (2pt);
\draw[color=blue] (-0.15,-0.27) -- (0,0);
\fill[color=blue] (-0.08,0.39) circle (2pt);
\draw[color=blue] (-0.08,0.39) -- (0,0);
\fill[color=blue] (1.33,-1.71) circle (2pt);
\draw[color=blue] (1.33,-1.71) -- (2,-2);
\fill[color=blue] (-2.27,-2.61) circle (2pt);
\draw[color=blue] (-2.27,-2.61) -- (-2,-2);
\fill[color=blue] (-2.6,0.3) circle (2pt);
\draw[color=blue] (-2.6,0.3) -- (-2,0);
\fill[color=blue] (1.92,-1.7) circle (2pt);
\draw[color=blue] (1.92,-1.7) -- (2,-2);
\fill[color=blue] (-1.3,0.16) circle (2pt);
\draw[color=blue] (-1.3,0.16) -- (-2,0);

\end{tikzpicture}
 \caption{Realization of network model}\label{toyFig} 
\end{figure}
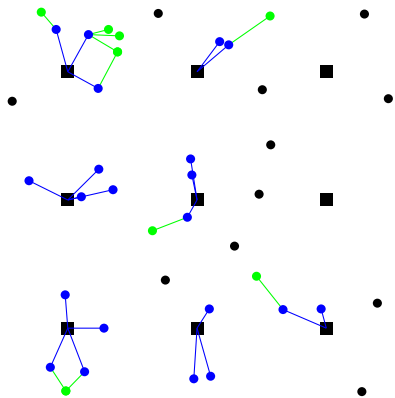

In the following, we write $H_r(x)$ for the smallest number $k\ge1$ such that $x\in\R^d$ is $k$-connectable to some point of $Y=rY^{(1)}$. The main object of investigation in this paper is the quantity
$$\Theta(k,r)=\lambda^{-1}\E\#\{X_i\in X\cap[-1/2,1/2]^d:\, H_r({X_i})\le k\},$$
i.e., the normalized expected number of points in $X\cap [-1/2,1/2]^d$ that are $k$-connectable to some base station. In fact, we show that $\Theta(k,r)$ admits a more natural representation as limiting quantity of the average number of points in $X$ inside a large box that are $k$-connectable to a point of $Y$. 
\begin{proposition}
\label{defEqProp}
Let $k\ge1$ and $r>0$. Then, almost surely,
$$\Theta(k,r)=\lim_{n\to\infty}\lambda^{-1}n^{-d}\#\{X_i\in X\cap[-n/2,n/2]^d: H_r({X_i})\le k\}.$$
\end{proposition}

 Provided that $k$ and $r$ are of the same order, the asymptotic behavior of $\Theta(k,r)$ depends sensitively on whether the intensity $\lambda$ is below or above the critical intensity $\lambda_c$ in continuum percolation. To be more precise, $\lambda_c$ is the infimum over all intensities $\lambda>0$ for which the union $\cup_{i=1}^\infty B_{1/2}(X_i)$ of balls of radius $1/2$ centered at points $X_i$ almost surely has an unbounded connected component.

In the sub-critical regime, $\Theta(k,r)$ decays polynomially in $r$ as $r\to\infty$.

\begin{theorem}
\label{subCrit}
Let $\lambda<\lambda_c$ and $r>0$. Then, 
$$\sup_{k\ge1}\Theta(k,r)\le\lambda^{-1}r^{-d}\lambda'\E\#C(o),$$ 
where $C(o)$ denotes the set of all $X_i\in X$ that are connectable to the origin.
\end{theorem}

Note that for $\lambda<\lambda_c$ we have $\E\#C(o)<\infty$, see e.g.~\cite[Theorem 12.35]{Grim99}.

Next, consider the supercritical case, i.e., let $\lambda>\lambda_c$. By a central result in continuum percolation~\cite[Theorem 2.1]{mRoy}, the set $\cup_{i=1}^\infty B_{1/2}(X_i)$ contains a unique unbounded connected component. In the following, $\mc{C}_\infty\subset X$ denotes the subset of all elements of $X$ that are contained in this unbounded connected component. We write $\theta$ for the probability that there exists $X_i\in\mc{C}_\infty$, with $|X_i|\le1$.

In order to describe the asymptotic behavior of $\Theta(k,r)$ for large $k$ and $r$, it is important to understand that the chemical distance between two points of $\mc{C}_\infty$, i.e., the minimum number of hops needed to establish a connection, grows linearly in the Euclidean distance of the two points. This can be formalized in different ways. 

First, fixing any point $X_i\in\mc{C}_\infty$, there should exist an a.s. finite random variable $\rho_i$ such that for every $X_j\in\mc{C}_\infty$ the chemical distance between $X_i$ and $X_j$ is at most $\rho_i|X_i-X_j|$. As observed in~\cite[Lemma 5.2]{bk1}, when considering Bernoulli site percolation on the lattice, the corresponding result can be derived by adapting the bond percolation argument established in Lemma 2.4 in the thesis of Antal~\cite{antalPhD}.

Additionally, when disregarding points in a small environment of $X_i$, the random variable $\rho_i$ can be replaced by a deterministic quantity $\mu\in(0,\infty)$ that does not depend on $i$. To be more precise, we put $q(x)=X_j$ if $X_j$ is the element of $\mc{C}_\infty$ minimizing the distance to $x\in\R^d$. Then, $D_n$ denotes the minimum integer $k\ge1$ such that $q(o)$ and $q(ne_1)$ are $k$-connectable, where $e_1=(1,0,\ldots,0)$ is the first standard unit vector in $\R^d$. Using Kingman's subadditive ergodic theorem, it is shown in~\cite{yao} that there exists a real number $\mu\in(0,\infty)$ such that almost surely, $\lim_{n\to\infty}n^{-1} D_n=\mu$; see also~\cite{antPiszt} for the corresponding statement on the lattice. 

With this background, we can now provide a heuristic explanation for the asymptotic behavior of $\Theta(k,r)$ if the speed at which $k$ and $r$ tend to infinity is chosen so that their quotient tends to some constant. To be more precise, by the Slivnyak-Mecke theorem~\cite[Corollary 3.2.3]{sWeil}, we have 
$$\Theta(k,r)=\P(r^{-1}{H_r(o)}{}\le r^{-1}{k}).$$
Hence, it suffices to understand the asymptotic distribution of $r^{-1}H_r=r^{-1}H_r(o)$ as $r\to\infty$. 
First, points of $X$ can only connect to points of $Y$ that are contained in the unbounded connected component of continuum percolation and the probability that a given point of $Y$ is contained in the unbounded connected component is given by $\theta$. Hence, instead of 
$rY^{(1)}$ we consider the process of relevant points $rY^{(\theta)}$, where $Y^{(\theta)}$ is obtained from $Y^{(1)}$ by independent thinning with survival probability $\theta$. Then, for a given point of $X$ to be connectable to some point of $Y$, the former must also belong to the unbounded connected component, which occurs with probability $\theta$. Moreover, the closest point of $rY^{(\theta)}$ is at Euclidean distance $r\min\{|y|:y\in Y^{(\theta)}\}$ and it can be reached in at most $\mu r\min\{|y|:y\in Y^{(\theta)}\}$ hops.
This heuristic is made precise in the following result, where we use the convention $0\cdot\infty=0$.

\begin{theorem}
\label{supCritDist}
Let $\lambda>\lambda_c$.
Then, $r^{-1}H_r$ converges in distribution to the random variable 
$$(1-Z)\cdot\infty+Z\mu\min\{|y|:y\in Y^{(\theta)}\},$$
where $Z$ is a Bernoulli random variable that is independent of $Y^{(\theta)}$ and which assumes the value $1$ with probability $\theta$.
\end{theorem}

In other words, the asymptotic distribution of $r^{-1}H_r$ is a mixture between a Dirac measure at $\infty$ and the contact distribution of the point process $\mu Y^{(\theta)}$. 
In particular, Theorem~\ref{supCritDist} can be used to compute $\lim_{r\to\infty}\P(H_r\le cr)$. 

\begin{corollary}
\label{supCritCor}
Let $\lambda>\lambda_c$ and assume that $\lim_{r\to\infty}r^{-1}k(r)=c$ for some $c\in(0,\infty)$. Then, 
$$\lim_{r\to\infty}\Theta(k,r)=\theta\P\Big(o\in\bigcup\nolimits_{Y_j\in Y^{(\theta)}}B_{c/\mu}(Y_j)\Big).$$
\end{corollary}
If $Y^{(1)}$ is a homogeneous Poisson point process with intensity $\lambda'\in(0,\infty)$, then $Y^{(\theta)}$ is again a homogeneous Poisson point process with intensity $\theta\lambda'$. In particular, we get the following result.

\begin{corollary}
\label{supCritCor1}
Let $Y^{(1)}$ be a homogeneous Poisson point process with intensity $\lambda'\in(0,\infty)$. Then, under the assumptions of Corollary~\ref{supCritCor},
$$\lim_{r\to\infty}\Theta(k,r)=\theta(1-\exp(\theta\kappa_dc^d\mu^{-d})),$$
where $\kappa_d$ denotes the volume of the unit ball in $\R^d$.
\end{corollary}

The limiting distribution provided in Theorem~\ref{supCritDist} depends on $\lambda$ implicitly via $\theta$ and $\mu$. In order to develop an intuition on the order of $\lambda$ that is needed to achieve a given (high) connectivity probability, it is useful to have some information on the behavior of $\theta$ and $\mu$ as a function of $\lambda$. First, concerning $\theta$, it is shown in~\cite[Corollary of Theorem 3]{pcPerc} that $\theta=\theta(\lambda)$ converges exponentially fast to $1$ as $\lambda$ tends to infinity. Second, we show that asymptotically $\mu-1=\mu(\lambda)-1$ tends to $0$ as $\lambda\to\infty$ and that the convergence occurs at least at a polynomial speed.
\begin{theorem}
\label{supCritCor3}
$\mu(\lambda)-1\in O(\lambda^{-1/d}(\log\lambda)^{1/d}).$
\end{theorem}

The present paper is organized as follows. In Section~\ref{subCritSec}, we establish the ergodic representation of $\Theta(k,r)$ announced in Proposition~\ref{defEqProp} and investigate the asymptotic behavior of $\Theta(k,r)$ in the subcritical regime. That is, we prove Theorem~\ref{subCrit}. Section~\ref{supCritSec}, is devoted to the proof of Theorem~\ref{supCritDist} which describes the distributional limit of the rescaled minimum number of hops $r^{-1}H_r$ in the supercritical regime. Finally, in Section~\ref{muAsySec}, we prove Theorem~\ref{supCritCor3}, i.e., we show that the time constant $\mu$ tends to $1$ as the intensity tends to infinity. Additionally,  we provide a lower bound for the speed of this convergence.

\section{Proof of Proposition~\ref{defEqProp} and Theorem~\ref{subCrit}}
\label{subCritSec}
The proof of Proposition~\ref{defEqProp} is based on the multidimensional ergodic theorem. To apply this result, it is important to note that the homogeneous Poisson point process is mixing~\cite[Theorem 9.3.5]{sWeil}, so that the pair of independent stationary point processes $(X,Y)$ is again ergodic, see~\cite[Theorem 3.6]{krengel}.
\begin{proof}
For $z\in\R^d$ let 
$$W_z=\#\{X_i\in [-1/2,1/2]^d+z:\, H_r(X_i)\le k\}$$ 
denote the number of points in $X\cap (z+[-1/2,1/2]^d)$ that are at most $k$ hops away from some point of $Y$. From the ergodic theorem for spatial processes (see, e.g.~\cite[Theorem 2.13]{krengel}), we conclude that the random variable
$$\Xi_m=m^{-d}\int_{[-m/2,m/2]^d}W_z\d z$$
converges almost surely to 
$$\E\int_{[-1/2,1/2]^d} W_z \d z=\E\#\{X_i\in [-1/2,1/2]^d:\, H_r(X_i)\le k\}.$$
Moreover, for sufficiently large $n\ge1$ the expression 
$$n^{-d}\#\{X_i\in X\cap[-n/2,n/2]^d: H_r({X_i})\le k\}$$
is bounded below and above by
$n^{-d}(n-1)^d \Xi_{n-1},$
and 
$n^{-d}(n+1)^d \Xi_{n+1},$
respectively.
 Hence, letting $n\to\infty$ completes the proof.
\end{proof}

To prepare the proof of Theorem~\ref{subCrit}, we note that it is possible to express $\Theta(k,r)$ as the expected value of the suitably weighted size of the cluster at a typical point of $Y$.
To be more precise, for $Y_j\in Y$, let $C_k(Y_j)$ denote the set of all $X_i\in X$ such that $X_i$ is $k$-connectable to $Y_j$.
Additionally, put 
$$\kappa(X_i)=\#\{Y_j\in Y:X_i\in C_k(Y_j)\}.$$
Then, we show that $\Theta(k,r)=\lambda^{-1}\E\sum_{Y_j\in[-1/2,1/2]^d}\sum_{X_i\in C_k(Y_j)}\kappa(X_i)^{-1}$.
\begin{lemma}
\label{mtpLem}
Let $k\ge1$ and $r>0$. Then,
$$\Theta(k,r)=\lambda^{-1}\E\sum_{Y_j\in[-1/2,1/2]^d}\sum_{X_i\in C_k(Y_j)}\kappa(X_i)^{-1}.$$
\end{lemma}
\begin{proof}
The claimed identity is a consequence of the mass-transport principle~\cite{groupPerc}. Indeed, define a function $\Phi:\Z^d\times\Z^d\to[0,\infty)$ by mapping a pair of sites $(z,z^\p)\in\Z^d\times\Z^d$ to
$$\Phi(z,z^\p)=\sum_{Y_j\in [-1/2,1/2]^d+z}\sum_{X_i\in C_k(Y_j)\cap ([-1/2,1/2]^d+z')}\kappa(X_i)^{-1}.$$
Then, clearly, $\sum_{z\in\Z^d} \Phi(o,z)=\sum_{Y_j\in [-1/2,1/2]^d}\sum_{X_i\in C_k(Y_j)}\kappa(X_i)^{-1}$. On the other hand, 
\begin{align*}
\sum_{z\in\Z^d} \Phi(z,o)&= \sum_{Y_j\in Y}\sum_{X_i\in C_k(Y_j)\cap [-1/2,1/2]^d}\kappa(X_i)^{-1}\\
&=\sum_{X_i\in [-1/2,1/2]^d}\sum_{\substack{Y_j\in Y: X_i\in C_k(Y_j)}}\kappa(X_i)^{-1}\\
&=\#\{X_i\in [-r/2,r/2]^d:\, X_i\text{ is $k$-connectable to some point of $Y$}\}.
\end{align*}
By stationarity, we obtain that
$$\E\sum_{z\in\Z^d}\Phi(z,o)=\sum_{z\in\Z^d}\E\Phi(z,o)=\sum_{z\in\Z^d}\E\Phi(o,-z)=\E\sum_{z\in\Z^d}\Phi(o,z),$$
which concludes the proof.
\end{proof}
Since $\kappa(X_i)\ge1$ for all $X_i\in C_k(Y_j)$, Lemma~\ref{mtpLem} gives rise to a simple upper bound for $\Theta(k,r)$. 
\begin{proposition}
\label{simpUpBound}
Let $k\ge1$ and $r>0$. Then,
$$\Theta(k,r)\le\lambda^{-1}r^{-d}\lambda'\E\#C_k(o).$$
\end{proposition}
We note two corollaries of Proposition~\ref{simpUpBound}. First, $k$ must grow at least linearly in $r$ for $\Theta(k,r)$ to have a non-zero limit.
\begin{corollary}
\label{smallKCor}
If $k=k(r)\in o(r)$, then $\lim_{r\to\infty}\Theta(k,r)=0$.
\end{corollary}
\begin{proof}
Since $C_k(o)$ is contained in $B_k(o)$, we deduce that $\E\#C_k(o)\le k^d\E\#(X\cap B_1(o))$. In particular, applying the upper bound from Proposition~\ref{simpUpBound} proves the claim. 
\end{proof}
Moreover, Proposition~\ref{simpUpBound} is also useful for proving Theorem~\ref{subCrit}. 

\begin{proof}[Proof of Theorem~\ref{subCrit}]
Combining the trivial inequality $\#C_k(o)\le\#C(o)$ with Proposition~\ref{simpUpBound} yields the desired bound.
\end{proof}

\section{Proof of Theorem~\ref{supCritDist}}
\label{supCritSec}
In this section, we prove Theorem~\ref{supCritDist}. To this end, we fix $\lambda>\lambda_c$ throughout the entire section. Using the notation of Theorem~\ref{supCritDist}, let $W=(1-Z)\cdot \infty+Z\mu\min\{|y|:y\in Y^{(\theta)}\}$. In order to show that $r^{-1}H_r$ converges to $W$ in distribution, we fix an arbitrary $a\ge0$. Then, we proceed in three steps, namely
\begin{enumerate}
\item $\lim_{r\to\infty}\P(H_r=\infty)=1-\theta$,
\item $\liminf_{r\to\infty}\P(H_r\le ra)\ge\P(W\le a)$,
\item $\limsup_{r\to\infty}\P(H_r\le ra)\le\P(W\le a)$.
\end{enumerate}

As a first auxiliary result, we note that asymptotically the events that points in $\R^d$ belong to the unbounded connected component become independent. 

\begin{lemma}
\label{asyIndLem}
Let $\lambda>\lambda_c$ and $z_1,\ldots,z_m$ be distinct points in $\R^d\setminus\{o\}$. Furthermore, let $E_{r}$ denote the event that $\#C(o)=\infty$ and $\#C(rz_i)=\infty$ for some $i\in\{1,\ldots,m\}$. Then, $\lim_{r\to\infty}\P(E_{r})=\theta (1-(1-\theta)^m).$
\end{lemma}
\begin{proof}
Choose $\delta>0$ such that the cubes $[-\delta,\delta]^d, z_1+[-\delta,\delta]^d,\ldots,z_m+[-\delta,\delta]^d$ are pairwise disjoint. Furthermore, let $G(y,r)$ denote the event that the connected component of $B_{1/2}(y)\cup\bigcup_{j\ge1}B_{1/2}(X_j)$ at $y\in\R^d$ is not contained in $y+[-r\delta+1,r\delta-1]^d$. Since the events $G(o,r),G(rz_1,r)\ldots,$ $G(rz_m,r)$ are independent, we can conclude that 
\begin{align*}
\lim_{r\to\infty}\P(E_{r})&=\lim_{r\to\infty}\P(G(o,r))\Big(1-\prod_{i=1}^m\big(1-\P(G(rz_i,r))\big)\Big)=\theta(1-(1-\theta)^m),
\end{align*}
if we can show that the probability that the connected component of $B_{1/2}(y)\cup\bigcup_{j\ge1}B_{1/2}(X_j)$ at $y$ is finite, but not contained in $y+[-r\delta+1,r\delta-1]^d$ tends to $0$ as $r\to\infty$. But this is a consequence of the uniqueness of the unbounded connected component in continuum percolation, see~\cite[Theorem 2.1]{mRoy}.
\end{proof}

Lemma~\ref{asyIndLem} allows us to compute $\lim_{r\to\infty}\P(H_r=\infty)$.
\begin{proof}[Proof of $\lim_{r\to\infty}\P(H_r=\infty)=1-\theta$]
First, we note that $\limsup_{r\to\infty}\P(H_r<\infty)\le\theta$. For the reverse inequality, let $n\ge1$ be arbitrary. Uniqueness of the infinite connected component shows that if $\#C(o)=\infty$ and $\#C(ry)=\infty$ for some $y\in Y^{(1)}\cap[-n/2,n/2]^d$, then $H_r<\infty$. Hence, by Fatou's lemma and Lemma~\ref{asyIndLem},
\begin{align*}
\liminf_{r\to\infty}\P(H_r<\infty)&\ge\E\Big(\liminf_{r\to\infty}\P(\text{$\#C(o)=\infty$ and $\sup_{y\in Y^{(1)}\cap[-n/2,n/2]^d}\#C(ry)=\infty$}|Y^{(1)}) \Big)\\
&=\theta\E(1-(1-\theta)^{\#(Y^{(1)}\cap [-n/2,n/2]^d)}).
\end{align*}
Letting $n\to\infty$ completes the proof of the lower bound.
\end{proof}

\begin{lemma}
\label{yaoLem}
Let $\varepsilon\in(0,1)$ be arbitrary. Then,
$$\lim_{r\to\infty}\P (E(r,\varepsilon))=0,$$
where $E(r,\varepsilon)$ denotes the event that there exists $y\in Y^{(1)}\cap B_{a(1-\varepsilon)/\mu}(o)$ such that $\#C(o)=\#C(ry)=\infty$, but $o$ is not $\lfloor ra\rfloor$-connectable to $ry$.
\end{lemma}
\begin{proof}
The claim is an immediate consequence of~\cite[Theorem 2.2]{yao}.
\end{proof}

After these preliminary results, we now proceed with the proof of $\liminf_{r\to\infty}\P(H_r\le ra)\ge\P(W\le a)$.
\begin{proof}[Proof of $\liminf_{r\to\infty}\P(H_r\le ra)\ge\P(W\le a)$]
Put $E^*(r,\varepsilon)=\{\#C(o)=\infty\}\cap E^{**}(r,\varepsilon)$, where $E^{**}(r,\varepsilon)$ denotes the event that there exists $y\in Y^{(1)}\cap B_{a(1-\varepsilon)/\mu,}(o)$ with $\#C(ry)=\infty$. Then,
$$\P(H_r\le ra)\ge\P(E^*(r,\varepsilon))-\P(E(r,\varepsilon)).$$
By Lemma~\ref{yaoLem}, the second probability in the above expression is negligible as $r\to\infty$. 
Hence, by Lemma~\ref{asyIndLem},
\begin{align*}
\liminf_{r\to\infty}\P(H_r\le ra)&\ge\theta \E \Big(1-(1-\theta)^{\#( Y^{(1)}\cap B_{a(1-\varepsilon)/\mu}(o))}\Big)\\
&=\theta\P\big(Y^{(\theta)}\cap B_{a(1-\varepsilon)/\mu}(o)\ne\es\big)\\
&=\theta\P\big(\mu\min\{|y|:y\in Y^{(\theta)}\}\le a(1-\varepsilon)\big).
\end{align*}
Letting $\varepsilon\to0$ completes the proof.
\end{proof}

In order to complete the proof of Theorem~\ref{supCritDist}, it remains to show that $\limsup_{r\to\infty}\P(H_r\le ra)\le\P(W\le a)$.
First, we derive an auxiliary result illustrating the close relationship between the Euclidean distance and the chemical distance in the unbounded connected component of continuum percolation~\cite{yao} to show that, asymptotically, users are not $k$-connectable to base stations that are not within distance of $k/\mu$. To be more precise, we use the following corollary to the shape theorem~\cite[Theorem 2.2]{yao}.

\begin{lemma}
\label{yaoLem2}
Let $\lambda>\lambda_c$ and $a>0$.  Then, for every $\varepsilon\in(0,1)$,
$$\lim_{r\to\infty}\P(F(r,\varepsilon))=0,$$
where $F(r,\varepsilon)$ is the event that the origin is $\lceil ra\rceil$-connectable to some point in $\R^d\setminus B_{ra(1+\varepsilon)/\mu}(o)$.
\end{lemma}

Second, we note that, asymptotically, distinct points that are connectable must be contained in the unbounded connected component of continuum percolation.
\begin{lemma}
\label{asyIndLem2}
Let $\lambda>\lambda_c$ and $z_1,\ldots,z_m$ be distinct points in $\R^d\setminus\{o\}$. Furthermore, let $F_r$ denote the event that $o$ is connectable to some $rz_i$ with $\min\{\#C(o),\#C(rz_i)\}<\infty$. Then, $\lim_{r\to\infty}\P(F_r)=0$.
\end{lemma}
\begin{proof}
Let $\delta$ be the minimum of the pairwise distances between elements of $\{o,z_1,\ldots,z_m\}$. By stationarity, $\P(F_r)$ is bounded above by two times the probability that $\#C(o)<\infty$, but the origin is connectable to some point with distance at least $r\delta$. By uniqueness of the unbounded connected component in continuum percolation, the probability of the latter event tends to $0$ as $r\to\infty$. 
\end{proof}

Now, we can complete the proof of Theorem~\ref{supCritDist}.

\begin{proof}[Proof of $\limsup_{r\to\infty}\P(H_r\le ra)\le\P(W\le a)$]
First, we see that $\P(H_r\le ra)$ is at most 
$$\P(F'(r,\varepsilon))+\P(F(r,\varepsilon)),$$
 where  $F'(r,\varepsilon)$ denotes the event that $o$ is connectable to some $ry\in rY^{(1)}\cap B_{ra(1+\varepsilon)/\mu}(o)$.
We conclude from Lemma~\ref{yaoLem2} that it suffices to investigate the first expression. 
Concerning $\P(F'(r,\varepsilon))$, Lemma~\ref{asyIndLem2} shows that as $r\to\infty$ this probability converges to the probability of the event that $\#C(o)=\infty$ and $\#C(ry)=\infty$ for some $y\in Y^{(1)}\cap B_{a(1+\varepsilon)/\mu}(o)$. Hence, combining Lemma~\ref{asyIndLem} with the dominated convergence theorem gives that 
\begin{align*}
\limsup_{r\to\infty}\P(H_r\le ra)&\le\theta\P(Y^{(\theta)}\cap B_{a(1+\varepsilon)/\mu}(o)\ne\es).
\end{align*}
Repeating the final steps used in the derivation of the lower bound completes the proof.
\end{proof}

\section{Proof of Theorem~\ref{supCritCor3}}
\label{muAsySec}
Loosely speaking, in order to prove Corollary~\ref{supCritCor3}, we can proceed similarly as in~\cite[Lemma 3.4]{yao} and modify the arguments used in the lattice setting~\cite{antPiszt}. The general construction presented in these papers is useful for the proof of Corollary~\ref{supCritCor3}, but the identification of the behavior of $\mu=\mu(\lambda)$ as $\lambda\to\infty$ requires a more refined analysis. 

It is convenient to introduce a specific family of site percolation processes. For this purpose, we describe certain useful configurations in the unit cube. Let $\varepsilon\in(0,1/d)$ be arbitrary.
First, we need to ensure that any two points of $X\cap[-(1-\varepsilon)/2,(1-\varepsilon)/2]^d$ can be connected via hops of distance at most 1 to other points of $X\cap [-(1-\varepsilon)/2,(1-\varepsilon)/2]^d$. To be more precise, $E_{1,\varepsilon}$ denotes the event consisting of all locally finite $\varphi\subset\R^d$ such that $\varphi\cap Q_i\ne\es$ for all $i\in\{1,\ldots, (2d)^d\}$, where $Q_1,\ldots, Q_{(2d)^d}$ is a subdivision of $[-(1-\varepsilon)/2,(1-\varepsilon)/2]^d$ into congruent subcubes of side length $(1-\varepsilon)/(2d)$. In particular, if $Q_i\cap Q_j\ne\es$, then $|x_i-x_j|\le1$ for all $x_i\in Q_i$, $x_j\in Q_j$. 

Second, we demand that $X$ has a point close to the origin. This will allow us to pass through linear arrangements of adjacent cubes without deviating too much from the line segment connecting the centers of these cubes. More precisely, $E_{2,\varepsilon}$ denotes the event consisting of all locally finite $\varphi\subset\R^d$ with $\varphi\cap[-\varepsilon/4,\varepsilon/4]^d\ne\es$. Note that $|x-y|\le1$ for all $x\in[-\varepsilon/4,\varepsilon/4]^d$ and $y\in((1-\varepsilon)e_1+[-\varepsilon/4,\varepsilon/4]^d)$. Finally, for $\varepsilon\in(0,1)$ we say that $z\in\Z^d$ is \emph{$\varepsilon$-good} if $X-(1-\varepsilon)z\in E_{1,\varepsilon}\cap E_{2,\varepsilon}$. 

To begin with, we show that we can traverse quickly linear arrangements of good sites.
\begin{lemma}
\label{linArrLem}
Let $j\ge1$ and $\varepsilon\in(0,1)$ be such that the site $ie_1$ is $\varepsilon$-good for all $i\in\{0,\ldots,j\}$. Furthermore, let $x,y\in X$ be such that $x\in[-\varepsilon/4,\varepsilon/4]^d$ and $y\in(j(1-\varepsilon)e_1+[-\varepsilon/4,\varepsilon/4]^d)$. Then, $x$ and $y$ are $j$-connectable.
\end{lemma}
\begin{proof}
Proceeding inductively, it suffices to consider the case $j=1$. But for $j=1$, the claim is immediate. Indeed, as observed above, we have $|x-y|\le1$.
\end{proof}

Even for large values of the intensity $\lambda$, the probability that the site $ie_1$ is $\varepsilon$-good for all $i\in\{0,\ldots,m\}$ decays exponentially fast in $m$. Therefore, we have to deal with the occasional occurrence of defects. In the following, we say that a set of sites $\Lambda\subset\Z^d$ is \emph{$*$-connected} if it forms a connected set in the graph whose vertices are given by $\Z^d$ and where $z,z'\in\Z^d$ are connected by an edge if $|z-z'|_\infty\le1$. We need a crude upper bound for the number of steps needed to traverse a set of cubes associated with a $*$-connected set of $\varepsilon$-good sites.
\begin{lemma}
\label{crudeBoundLem}
Let $\varepsilon>0$ and $\Lambda\subset\Z^d$ be a finite $*$-connected set of $\varepsilon$-good sites. Furthermore, let $x,x'\in X$ be such that $x\in(1-\varepsilon)(z+[-1/2,1/2]^d)$, $x'\in(1-\varepsilon)(z'+[-1/2,1/2]^d)$ for some $z,z'\in\Lambda$. Then $x$ and $x'$ are $k$-connectable for $k=(3+(2d)^d)\#\Lambda$.
\end{lemma}
\begin{proof}
If $z=z'$, then the definition of $\varepsilon$-goodness implies that $x$ and $x'$ are $k'$-connectable for $k'=2+(2d)^d$. Next, if $z,z'$ are such that $|z-z'|_\infty\le1$, then, again by the definition of $\varepsilon$-goodness, there exist $y,y'\in X$ with $y\in (1-\varepsilon)(z+[-1/2,1/2]^d)$, $y'\in (1-\varepsilon)(z'+[-1/2,1/2]^d)$, and $|y-y'|\le1$.
Hence, the proof of Lemma~\ref{crudeBoundLem} is completed by an elementary induction argument on the length of the path in $\Lambda$ connecting $z$ and $z'$.
\end{proof}
The next step is to combine Lemmas~\ref{linArrLem} and~\ref{crudeBoundLem} into an upper bound that is useful in situations where the $*$-connected $\varepsilon$-bad components associated with the sites $ie_1$, $i\in\{0,\ldots,m\}$ only cover a small proportion of these sites. More precisely, let $U_m$ be the union of the $*$-connected $\varepsilon$-bad components associated with the sites $ie_1$, $i\in\{0,\ldots,m\}$. If $ie_1$ is $\varepsilon$-good, then we define its $*$-connected $\varepsilon$-bad component to be empty. Note that $U_m$ is almost surely finite provided that $\lambda$ is sufficiently large. 

Let $U_m^{(\infty)}$ denote the unbounded connected component of $\Z^d\setminus U_m$. Then, $U_m'=\Z^d\setminus U_m^{(\infty)}$ consists of $m'\ge1$ $*$-connected components $U_m^{(1)},\ldots, U_m^{(m')}$. Let $\partial U_m^{(i)}$ denote the \emph{outer boundary} of $U_m^{(i)}$, i.e., $\partial U_m^{(i)}$ consists of all $z\in\Z^d\setminus U_m^{(i)}$ such that $|z-z'|_\infty=1$ for some $z'\in U_m^{(i)}$. Note that $\partial U_m^{(i)}$ is $*$-connected, since the outer boundary of any $*$-connected set is again $*$-connected, see~\cite[Lemma 2.23]{aofpp} (related results can be found in~\cite{deu96,bCon}).

Next, we identify subsets of $\{o,e_1,\ldots,me_1\}$ that form linear arrangements of $\varepsilon$-good sites. To be more precise, we construct two finite increasing subsequences $(a_{i})_{1\le i\le m''}$ and $(b_{i})_{1\le i\le m''}$ of $\{0,\ldots,m\}$ inductively as follows. If $\{o,e_1,\ldots,me_1\}\subset U_m'$, then we put $m''=0$. Otherwise, choose $a_1=\min\{i\ge0:ie_1\not\in U_m'\}$ as the first site that is not contained in $U_m'$. Furthermore, let $b_1=\max\{i\in\{a_1,\ldots,m\}:ie_1\not\in U_m'\}$ be the last site after $a_1$ that is not contained $U_m'$. If $b_1=m$, then put $m''=1$ and terminate the construction. Otherwise, by definition, there is some $i_1\in\{1,\ldots,m'\}$ such that $(b_1+1)e_1\in U_m^{(i_1)}$. Define $a_2'=\max\{i\ge b_1:ie_1\in\partial U_m^{(i_1)}\}$. If $a_2'>m$, then put $m''=1$ and terminate the construction. Otherwise, define $a_2=a_2'$ and continue inductively. See Figure~\ref{gamkLEFig} for an illustration of this construction.

\begin{figure}[!htpb]
\centering
\begin{tikzpicture}[scale=1.0]
\fill[black!30!white] plot [smooth cycle,thick,tension=0.5] coordinates {(-1,0) (-0.7,0.7) (1,0) (0.2,-1)};
\fill[black!30!white] plot [smooth cycle,thick,tension=0.6] coordinates {(4,0) (4.4,0.9) (5.8,0) (4.7,-0.8)};
\fill[black!30!white] plot [smooth cycle,thick,tension=0.5] coordinates {(1.5,0) (2.5,1.1) (3.5,0) (3.3,-0.5) (2.5,0.5) (1.8,-0.3)};
\coordinate[label=-90:\small{${0}$}] (u) at (0,0);
\coordinate[label=-90:\small{$a_1$}] (u) at (1,0);
\coordinate[label=-90:\small{$b_1$}] (u) at (1.5,0.05);
\coordinate[label=-90:\small{$a_2$}] (u) at (3.5,0);
\coordinate[label=-90:\small{$b_2$}] (u) at (4.1,0.05);
\coordinate[label=-90:\small{$m$}] (u) at (5,0);
\coordinate[label=-90:\small{$U_m^{(1)}$}] (u) at (-0.5,0.5);
\coordinate[label=-90:\small{$U_m^{(2)}$}] (u) at (2.5,1);
\coordinate[label=-90:\small{$U_m^{(3)}$}] (u) at (4.5,0.8);
\draw[thick] (0,0)--(5,0);
\fill (0,0) circle (2pt);
\fill (1,0) circle (2pt);
\fill (1.5,0) circle (2pt);
\fill (3.5,0) circle (2pt);
\fill (4.0,0) circle (2pt);
\fill (5,0) circle (2pt);
\end{tikzpicture}
\caption{Construction of the sequences $(a_i)_{1\le i\le m''}$ and $(b_i)_{1\le i\le m''}$}\label{gamkLEFig}
\end{figure}
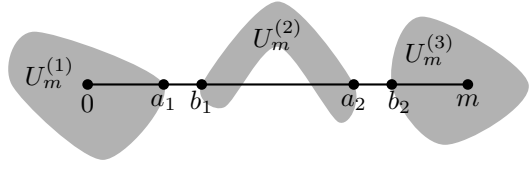

We make two crucial observations. First, the sites $je_1$ are $\varepsilon$-good for all $j\in\{a_i,\ldots,b_i\}$ and $i\in\{1,\ldots,m''\}$. Second, if $j<m'$ then the sites $b_je_1,a_{j+1}e_1$ are contained in the $*$-connected set $\partial U^{(i_j)}_m$. This allows us to make use of Lemma~\ref{crudeBoundLem}.

To summarize, we have derived bounds on the number of hops for traversing linear arrangements of $\varepsilon$-good cubes and for making detours around defects. These bounds are sufficient for our purposes provided that neither $o$ nor $me_1$ are contained in $U'_m$. In that situation, we need the following auxiliary result, where we write $\oplus$ for Minkowski addition.

\begin{lemma}
\label{inDefLem}
Let $i\in\{1,\ldots,m'\}$ and $x\in\mc{C}_\infty$ be such that $x\in(1-\varepsilon)(U^{(i)}_m\oplus[-1/2,1/2]^d)$. Then, there exists $x'\in (1-\varepsilon)( \partial U^{(i)}_m\oplus [-1/2,1/2]^d)$ such that $x$ and $x'$ are $k$-connectable for $k=c_1\#U_m^{(i)}$, where $c_1=c_1(d)\ge1$ is a constant depending only on the dimension $d$.
\end{lemma}
\begin{proof}
Loosely speaking, we proceed as follows. Since $x$ is contained in $\mc{C}_\infty$, it is $k$-connectable to the boundary of $(1-\varepsilon)(U_m^{(i)}\oplus[-1/2,1/2]^d)$ for some $k\ge1$. Then, we make use of the observation in~\cite[Lemma 3.4]{yao} that the minimum such $k$ cannot be too large in comparison to $\#U^{(i)}_m$. To be more precise, let $\gamma=\langle x=x_1,\ldots,x_k\rangle$ be some path in $X$ consisting of hops of distance at most $1$ such that $x'=x_k$ is contained in $(1-\varepsilon)(\partial U_m^{(i)}\oplus [-1/2,1/2]^d)$. We note that there is a constant $c_1'=c_1'(d)\ge1$ with the following property. There exists a finite subset $S$ of $\R^d$ consisting of at most $c_1'\#U_m^{(i)}$ elements and such that for every $y\in (1-\varepsilon)((U_m^{(i)}\cup \partial U_m^{(i)})\oplus [-1/2,1/2]^d)$ there exists $y'\in S$ with $|y-y'|\le 1/2$. If there exist $y_1,\ldots,y_k\in S$ with $|x_j-y_j|\le 1/2$ for every $j\in\{1,\ldots,k\}$ and such that for every $j\in\{1,\ldots,k\}$ there exists at most one $j'\in\{1,\ldots,k\}\setminus\{j\}$ with $y_j=y_{j'}$, then the claim follows from the observation that $k\le2\#S\le 2c_1'\#U_m^{(i)}$. Hence, it remains to transform $\gamma$ into a $\gamma'$ path with that property. This can be achieved by using Lawler's method of loop erasure~\cite{saw}.

To be more precise, let $i_1\in\{1,\ldots,k\}$ be the largest index such that $|x_{i_1}-y_1|\le1/2$. In particular, $|x_1-x_{i_1}|\le1$ and $|x_{i_1}-x_{i_1+1}|\le1$. Now the construction proceeds inductively by defining $\gamma'$ as the path obtained by pasting the paths $\langle x_1,x_{i_1},x_{i_1+1}\rangle$ and $\gamma''$, where $\gamma''$ is the loop erasure of the path $\langle x_{i_1+1},\ldots,x_k\rangle$. 
\end{proof}

Let $m_\varepsilon(n)$ be the unique integer contained in the interval $[\tfrac{n}{1-\varepsilon}-\tfrac{1}{2},\tfrac{n}{1-\varepsilon}+\tfrac{1}{2})$. Combining Lemmas~\ref{linArrLem}--\ref{inDefLem}, we see that $q(o)$ and $q(ne_1)$ can be connected using at most 
\begin{align}
\label{dnEq}
k=m_\varepsilon(n)+(3+(2d)^d)\sum_{i=1}^{m'}\#\partial U^{(i)}_{m_\varepsilon(n)}+2c_1\#U'_{m_\varepsilon(n)}
\end{align}
hops. 
In order to translate this observation into an upper bound for $\mu$, it is important to have some control on the size of the random variables $\sum_{i=1}^{m'}\#\partial U^{(i)}_{m_\varepsilon(n)}$ and $\#U'_{m_\varepsilon(n)}$. In the following, we write $q_{\lambda,\varepsilon}$ for the probability that a fixed site is $\varepsilon$-bad. In particular, 
\begin{align}
\label{qUpBound}
q_{\lambda,\varepsilon}\le (2d)^d\exp(-\lambda(1-\varepsilon)^d (2d)^{-d})+\exp(-\lambda2^{-d}\varepsilon^{d}).
\end{align}

\begin{lemma}
\label{clustSizeLem}
If $q_{\lambda,\varepsilon}<2^{-3^d-1}$, then $\lim_{m\to\infty}\P(\sum_{i=1}^{m'}\#\partial U^{(i)}_{m}\ge 2^{3^d+2}3^dq_{\lambda,\varepsilon}m)=0$.
\end{lemma}
\begin{proof}
Since any site in $\cup_{i=1}^{m'}\partial U^{(i)}_{m}$ is $*$-adjacent to an $\varepsilon$-bad $*$-connected component intersecting $\{o,e_1,\ldots,me_1\}$, we have that
$$\sum_{i=1}^{m'}\#\partial U^{(i)}_{m}\le3^d\#U_{m}.$$
Furthermore, as shown in~\cite[Lemma 2.3]{deu96}, $\#U_{m}$ is stochastically dominated by $\sum_{i=0}^mR_i$, where $\{R_i\}_{0\le i\le m}$ is a family of iid random variables such that $R_i$ has the distribution of the size of the open $*$-connected component at the origin when considering Bernoulli site percolation with parameter $q_{\lambda,\varepsilon}$. The number of $*$-connected subsets of sites containing the origin and consisting of exactly $k\ge1$ sites is bounded above by $2^{3^dk}$, see~\cite[Lemma 9.3]{penrose}. Therefore, 
$$\E R_0\le \sum_{k=0}^{\infty}k2^{3^dk}q^k_{\lambda,\varepsilon}=\frac{2^{3^d}q_{\lambda,\varepsilon}}{(1-2^{3^d}q_{\lambda,\varepsilon})^2}<2^{3^d+2}q_{\lambda,\varepsilon}.$$
The claim now follows from the law of large numbers.
\end{proof}

\begin{lemma}
\label{clustSizeLem2}
If $q_{\lambda,\varepsilon}<2^{-3^d-1}$, then $\lim_{m\to\infty}\P(\#U'_{m}\ge 2^{3^d+4}3^{3d}d^2q_{\lambda,\varepsilon}m)=0$.
\end{lemma}
\begin{proof}
By the isoperimetric inequality~\cite[Equation (2.1)]{deu96}, we have $\#U^{(i)}_m\le3^dd^2\big(\#\partial U^{(i)}_m\big)^2$ for all $i\in\{1,\ldots,m'\}$. Note that the factor $3^d$ is needed, since we consider outer boundaries with respect to $*$-adjacency.
Moreover, using the same notation as in the proof of Lemma~\ref{clustSizeLem}, the sum $\sum_{i=1}^{m'}\big(\#\partial U^{(i)}_m\big)^2$ is stochastically dominated by $9^{d}\sum_{i=0}^{m}R_i^2$, where 
$$\E R_0^2\le \sum_{k=0}^{\infty}k^22^{3^dk}q^k_{\lambda,\varepsilon}=\frac{(2^{3^d}q_{\lambda,\varepsilon}+1)2^{3^d}q_{\lambda,\varepsilon}}{(1-2^{3^d}q_{\lambda,\varepsilon})^3}<2^{3^d+4}q_{\lambda,\varepsilon}.$$
As before, the law of large numbers now implies the claim. 
\end{proof}

In order to prove Theorem~\ref{supCritCor3}, we need to decrease $\varepsilon$ accordingly in the size of $\lambda$. By the upper bound on $q_{\lambda,\varepsilon}$ derived in~\eqref{qUpBound}, we conclude that if we choose 
\begin{align}
\label{varepsEq}
\varepsilon=\varepsilon(\lambda)=2\lambda^{-1/d}(\log \lambda)^{1/d},
\end{align}
 then $\lim_{\lambda\to\infty}\varepsilon^{-1}q_{\lambda,\varepsilon}=0$.
\begin{proof}[Proof of Theorem~\ref{supCritCor3}]
Choose $\varepsilon$ as in~\eqref{varepsEq} and put $\mu^+=1+3\varepsilon$. Then, it suffices to show that $\P(D_n \ge n\mu^+)\to0$ as $n\to\infty$. Combining~\eqref{dnEq} with Lemmas~\ref{clustSizeLem} and~\ref{clustSizeLem2}, we see that it suffices to show that $m_\varepsilon(n)\le n(1+2\varepsilon)$. But since $1/(1-\varepsilon)<1+2\varepsilon$, this is an immediate consequence of the definition of $m_\varepsilon(n)$.
\end{proof}

\subsection*{Acknowledgments}
This research was supported by the Leibniz group on \emph{Probabilistic Methods for Mobile Ad-Hoc Networks}. The author thanks W.~K\"onig for introducing him to the model of bounded-hop percolation and for the encouragement to investigate the asymptotic behavior of $\Theta(k,r)$. Furthermore, the author is also grateful for many helpful discussions and remarks on earlier versions of the manuscript. 
\bibliography{}

\begin{thebibliography}{10}

\bibitem{antalPhD}
P.~Antal.
\newblock {\em Trapping {P}roblems for the {S}imple {R}andom {W}alk}.
\newblock PhD thesis, ETH Z\"urich, 1994.

\bibitem{antPiszt}
P.~Antal and {\'A}.~Pisztora.
\newblock On the chemical distance for supercritical {B}ernoulli percolation.
\newblock {\em Ann. Probab.}, 24(2):1036--1048, 1996.

\bibitem{groupPerc}
I.~Benjamini, R.~Lyons, Y.~Peres, and O.~Schramm.
\newblock Group-invariant percolation on graphs.
\newblock {\em Geom. Funct. Anal.}, 9(1):29--66, 1999.

\bibitem{bk1}
M.~Biskup and W.~K{\"o}nig.
\newblock Long-time tails in the parabolic {A}nderson model with bounded
  potential.
\newblock {\em Ann. Probab.}, 29(2):636--682, 2001.

\bibitem{deu96}
J.-D. Deuschel and {\'A}.~Pisztora.
\newblock Surface order large deviations for high-density percolation.
\newblock {\em Probab. Theory Related Fields}, 104(4):467--482, 1996.

\bibitem{Grim99}
G.~R. Grimmett.
\newblock {\em Percolation}.
\newblock Springer, New York, second edition, 1999.

\bibitem{aofpp}
H.~Kesten.
\newblock Aspects of first passage percolation.
\newblock In P.~L. Hennequin, editor, {\em \'{E}cole d'\'et\'e de
  probabilit\'es de {S}aint-{F}lour, {XIV}}, volume 1180 of {\em Lecture Notes
  in Mathematics}, pages 125--264. Springer, Berlin, 1986.

\bibitem{krengel}
U.~Krengel.
\newblock {\em Ergodic Theorems}.
\newblock Walter de Gruyter \& Co., Berlin, 1985.

\bibitem{saw}
G.~F. Lawler.
\newblock A self-avoiding random walk.
\newblock {\em Duke Math. J.}, 47(3):655--693, 1980.

\bibitem{mRoy}
R.~Meester and R.~Roy.
\newblock Uniqueness of unbounded occupied and vacant components in {B}oolean
  models.
\newblock {\em Ann. Appl. Probab.}, 4(3):933--951, 1994.

\bibitem{pcPerc}
M.~D. Penrose.
\newblock On a continuum percolation model.
\newblock {\em Adv. in Appl. Probab.}, 23(3):536--556, 1991.

\bibitem{penrose}
M.~D. Penrose.
\newblock {\em Random {G}eometric {G}raphs}.
\newblock Oxford University Press, Oxford, 2003.

\bibitem{sWeil}
R.~Schneider and W.~Weil.
\newblock {\em Stochastic and Integral Geometry}.
\newblock Springer, Berlin, 2008.

\bibitem{bCon}
{\'A}.~Tim{\'a}r.
\newblock Boundary-connectivity via graph theory.
\newblock {\em Proc. Amer. Math. Soc.}, 141(2):475--480, 2013.

\bibitem{yao}
C.-L. Yao, G.~Chen, and T.-D. Guo.
\newblock Large deviations for the graph distance in supercritical continuum
  percolation.
\newblock {\em J. Appl. Probab.}, 48(1):154--172, 2011.

\end{thebibliography}
\bibliographystyle{abbrv}
\end{document}